\documentclass[a4paper, 12pt]{article}
\usepackage{amsmath}
\usepackage{amssymb}
\usepackage{graphicx}
\usepackage{multirow}
\usepackage{float}
\usepackage{mathtools}
\usepackage{dsfont}
\usepackage{mathrsfs}
\makeatletter
\newcommand{\specificthanks}[1]{\@fnsymbol{#1}}
\makeatother
\newcommand\blfootnote[1]{%
	\begingroup
	\renewcommand\thefootnote{}\footnote{#1}%
	\addtocounter{footnote}{-1}%
	\endgroup
}

\graphicspath{ {PASTA DAS IMAGENS/}}
\title{Local Lift Dependence Scale}
\date{}
\author{Diego Marcondes\thanks{Instituto de Matem\'atica e Estat\'istica, Universidade de S\~ao Paulo, Brazil}
	\textsuperscript{, \specificthanks{2}}\blfootnote{E-mail: dmarcondes@ime.usp.br}\blfootnote{D. Marcondes has received financial support from CNPq during the development of this paper.} \and Adilson Simonis\footnotemark[1]}
\allowdisplaybreaks
\newtheorem{theorem}{Theorem}

\newtheorem{lemma}{Lemma}
\newtheorem{proposition}{Proposition}
\newtheorem{remark}{Remark}
\newtheorem{example}{Example}
\newenvironment{proof}{\paragraph{Proof:}}{\hfill$\square$}

\begin{document}
	\maketitle
	
	\begin{abstract}
		We propose a local and general dependence quantifier between two random variables $X$ and $Y$, which we call Local Lift Dependence Scale, that does not assume any form of dependence (e.g., linear) between $X$ and $Y$, and is defined for a large class of random variables, singular and absolutely continuous w.r.t Lebesgue measure. We argue that this dependence scale is more general and suitable to study variable dependence than other specific local dependence quantifiers and global dependence coefficients, as the Mutual Information. An outline of how this dependence scale may be useful in branches of applied probability and topics for future research are presented.
		
	\end{abstract}
	
	\section{Introduction}
	\label{Sec1}
	
	The Mutual Information between two random variables $X$ and $Y$, defined in $(\Omega,\mathcal{F},\mathbb{P})$, is given by
	\begin{equation*}
	I(X,Y) = \mathbb{E}\Bigg[\log \frac{d\mathbb{P}(X,Y)}{d\big(\mathbb{P}(X) \times \mathbb{P}(Y)\big)} \Bigg]	
	\end{equation*}
	in which $d\mathbb{P}(X,Y)/d(\mathbb{P}(X) \times \mathbb{P}(Y))$ is the Radon-Nikodym derivative of the joint measure $\mathbb{P}(X,Y)$ with respect to the product measure $\mathbb{P}(X) \times \mathbb{P}(Y)$, when $\mathbb{P}(X,Y) \ll \mathbb{P}(X) \times \mathbb{P}(Y)$. The Mutual Information quantifies the mass concentration of the joint distribution of $X$ and $Y$ and was first proposed for random variables absolutely continuous w.r.t. counting measure by \cite{shannon}. As more dependent the variables are, the more concentrated is the mass on their joint distribution, so that Mutual Information also quantifies dependence. Indeed, the Mutual Information assesses \textit{any} kind of dependence, which makes it a global and general dependence coefficient.
	
	Although the Mutual Information is quite general, it does not detail the form of dependence between the variables as, for example, even though a high value of $I(X,Y)$ implies that $X$ and $Y$ are \textit{highly dependent}, it does not present any evidences about the \textit{kind} of dependence being observed. Therefore, the Mutual Information as a general and global dependence coefficient is not enough to assess all the nuances of the dependence between variables, so that it is necessary to apply other tools in order to not only point out the existence of a dependence, but also characterize the \textit{kind} of dependence being observed. 
	
	An important tool for assessing the form of dependence between two random variables are local dependence quantifiers which associate a value to each point in the support of the variables assessing their pointwise dependence. Local dependence quantifiers give a wide view of the relation between random variables, without summarizing it to an index. As examples of local dependence quantifiers we have the measure of linear local dependence of \cite{bairamov}, the curve of correlation of \cite{bjerve} and the function of local dependence of \cite{holland}. Another local dependence quantifier is the \textit{Sibuya's dependence function}.
	
	The dependence between random variables may be expressed by \textit{Sibuya's dependence function} $\Omega^{*}(x,y)$, proposed by \cite{sibuya}, and given by the relation $F(x,y) = \Omega^{*}(x,y)G(x)H(y)$, in which $F$ is the joint and $G$ and $H$ are the respective marginal cumulative distributions of two random variables $X$ and $Y$ defined in a same probability space. From $\Omega^{*}(x,y)$ it may be established, for example, if the random variables are independent, i.e., $\Omega^{*} \equiv 1$. However, as $\Omega^{*}(x,y)$ expresses the relation between the probability of the events $\{X \leq x,Y \leq y\}$, $\{X \leq x\}$ and $\{Y \leq y\}$, it is not sharp enough to characterize local dependence meaningfully. On the other hand, in order to asses the dependence being outlined by the Mutual Information, we may study a local dependence quantifier given by the Lift Function which is defined as the Radon-Nikodym derivative
	\begin{align}
	\label{lift_introduction}
	L(x,y) = \frac{d\mathbb{P}(X,Y)}{d\big(\mathbb{P}(X) \times \mathbb{P}(Y)\big)}(x,y)
	\end{align} 
	when it is defined, and then $I(X,Y) = \mathbb{E}[ \log L(X,Y)]$. 
	
	The main interest in studying the Lift Function is in determining for which values of $(x,y)$ $L(x,y) < 1$, for which $L(x,y) = 1$ and for which $L(x,y) > 1$. Indeed, when $(X,Y)$ is absolutely continuous w.r.t. counting measure, $L(x,y) < 1$ implies $\mathbb{P}(X = x \mid Y = y) < \mathbb{P}(X = x)$ and $\mathbb{P}(Y = y \mid X = x) < \mathbb{P}(Y = y)$, and we say that $\{Y = y\}$ \textit{inhibits} $\{X = x\}$, for the fact that $Y$ being equal to $y$ decreases the probability of $X$ being equal to $x$, and vice versa ($\{X = x\}$ \textit{inhibits} $\{Y = y\}$). Analogously, we have that $L(x,y) > 1$ implies $\mathbb{P}(X = x \mid Y = y) > \mathbb{P}(X = x)$ and $\mathbb{P}(Y = y \mid X = x) > \mathbb{P}(Y = y)$, and we say that $\{Y = y\}$ \textit{lifts} $\{X = x\}$, for the fact that $Y$ being equal to $y$ increases the probability of $X$ being equal to $x$, and vice versa ($\{X = x\}$ \textit{lifts} $\{Y = y\}$). For simplicity, we call $L$ the Lift Function and, for a fixed pair $(x,y)$, it may be interpreted from both sides: as $Y$ lifting $X$ or $X$ lifting $Y$. When $L(x,y) = 1$ we have that $\mathbb{P}(X = x \mid Y = y) = \mathbb{P}(X = x)$ and $\mathbb{P}(Y = y \mid X = x) = \mathbb{P}(Y = y)$, so that knowing that $\{X = x\}$ does not change the probability of $\{Y = y\}$, and vice versa. 
	
	From the study of the Lift Function it is possible to understand the dependence outlined by $I(X,Y)$ by observing the \textit{lifting pattern} of $X$ and $Y$. Furthermore, even if $X$ and $Y$ are not absolutely continuous w.r.t. Lebesgue measure, nor w.r.t. counting measure, the Lift Function still represents the dependence between $X$ and $Y$, as it compares pointwise the joint distribution of $X$ and $Y$ with their distribution in case they were independent. Thus, the Mutual Information and the Lift Function together present a wide view of the dependence between random variables: while the former assesses the existence of a dependence, the latter expresses the kind of dependence being observed. Therefore, in order to understand the dependence of a joint probability distribution, one can quantify it globally and locally to capture all its nuances. 
	
	However, the Lift Function is not well defined for all random variables $(X,Y)$, since if $\mathbb{P}(X,Y) \not\ll \mathbb{P}(X) \times \mathbb{P}(Y)$, then the Radon-Nikodym derivative (\ref{lift_introduction}) is not defined. Indeed, this is an issue, for there are some cases of interest such that there is no definition for the Lift Function. For example, let $\{(X_{n},Y_{n})\}_{n \geq 1}$ be random variables with the Standard Bivariate Normal Distribution with correlation $r_{n}$, such that $r_{n} \xrightarrow{n \rightarrow \infty} 1$. Then, $(X_{n},Y_{n})$ converges in law to $(X,Y)$, in which $X$ has the Standard Normal Distribution and $Y = X$ with probability $1$, but the Lift Function is not defined for $(X,Y)$. In fact, if $Y = \varphi(X)$, in which $X$ is absolutely continuous w.r.t. Lebesgue measure and $\varphi$ is a real-valued smooth function, then definition (\ref{lift_introduction}) of the Lift Function does not apply to $(X,Y)$.
	
	In this paper we propose a general definition for the Lift Function that contemplates a larger class of distributions, even when $\mathbb{P}(X,Y) \not\ll \mathbb{P}(X) \times \mathbb{P}(Y)$, and which reduces to (\ref{lift_introduction}) when $\mathbb{P}(X,Y) \ll \mathbb{P}(X) \times \mathbb{P}(Y)$. Furthermore, we argue that the Mutual Information may not be a suitable scale for measuring variable dependence, since, if $(X_{n},Y_{n})$ converges in law to $(X,Y)$, then $I(X_{n},Y_{n})$ may not converge to $I(X,Y)$. In fact, $Y$ may be equal to $\varphi(X)$ with probability $1$, but $I(X,Y) < I(X_{n},Y_{n})$ for infinitely many $n$. This fact shows that in order to understand the form of dependence between two variables it is necessary to study it locally, and in general, without restricting our attention to an specific kind of dependence.
	
	In Section \ref{Sec2} we define the Lift Function and study its properties when $\mathbb{P}(X,Y) \ll \mathbb{P}(X) \times \mathbb{P}(Y)$. In Section \ref{Sec3} we develop the Lift Function for the case in which $\mathbb{P}(X)$ and $\mathbb{P}(Y)$ are absolutely continuous w.r.t. Lebesgue measure in $\mathbb{R}$, but $\mathbb{P}(X,Y) \not\ll \mathbb{P}(X) \times \mathbb{P}(Y)$. Our generalization is closely related to the Hausdorff dimension of the support of the singular part of $\mathbb{P}(X,Y)$. In Section \ref{Sec4} we argue that the Mutual Information is not enough to assess the dependence between two random variables, and in Section \ref{Sec5} we present our final remarks and topics for future research. 
	
	\section{Lift Function}
	\label{Sec2}
	
	Let $(X,Y)$ be real-valued random variables defined in $(\Omega,\mathcal{F},\mathbb{P})$ and let $(\mathbb{R},\beta_{\mathbb{R}})$ and $(\mathbb{R}^{2},\beta_{\mathbb{R}^{2}})$ be the usual measurable spaces, in which $\beta_{\mathbb{R}}$ and $\beta_{\mathbb{R}^{2}}$ are the Borelians of $\mathbb{R}$ and $\mathbb{R}^{2}$, respectively. Let $\mu$ be a $(\mathbb{R}^{2},\beta_{\mathbb{R}}^{2})$-probability measure such that $\mu(A) = \mathbb{P}((X,Y) \in A), \forall A \in \beta_{\mathbb{R}^{2}}$. Define $\mu_{X}$, $\mu_{Y}$ as $\mu_{X}(B) = \mu(B \times \mathbb{R}), \mu_{Y}(B) = \mu(\mathbb{R} \times B), \forall B \in \beta_\mathbb{R}$, and $\mu_{XY}$ as the only $(\mathbb{R}^{2},\beta_{\mathbb{R}^{2}})$-probability measure which satisfies $\mu_{XY}(A \times B) = \mu_{X}(A)\mu_{Y}(B)$ for all $A, B \in \beta_{\mathbb{R}}$. 
	
	The probability measure $\mu$ in $(\mathbb{R}^{2},\beta_{\mathbb{R}^{2}})$ is induced by measurable vector function $(X,Y)$; $\mu_{XY}$ is the product measure induced in $(\mathbb{R}^{2},\beta_{\mathbb{R}^{2}})$ by measurable functions $X$ and $Y$; and $\mu_{X}$ and $\mu_{Y}$ are the measures induced in $(\mathbb{R},\beta_{\mathbb{R}})$ by measurable functions $X$ and $Y$, respectively. Throughout this section we suppose that $\mu \ll \mu_{XY}$. We denote $\mathscr{L}_{1}$ and $\mathscr{L}_{2}$ the Lebesgue measures in $(\mathbb{R},\beta_{\mathbb{R}})$ and $(\mathbb{R}^{2},\beta_{\mathbb{R}^{2}})$, respectively.
	
	We define a Lift Function $L_{(X,Y,\mathbb{P})}: \mathbb{R}^{2} \rightarrow \mathbb{R}_{+}$ of $(X,Y)$ under $\mathbb{P}$ as the Radon-Nikodym derivative of $\mu$ with respect to $\mu_{XY}$, that is
	\begin{align*}
	L_{(X,Y,\mathbb{P})}(x,y) = \frac{d \mu}{d\mu_{XY}}(x,y) & & (x,y) \in \mathbb{R}^{2}
	\end{align*}
	which is a $\beta_{\mathbb{R}^{2}}$-measurable function that satisfies
	\begin{align}
	\label{LF}
	\int\displaylimits_{A} L_{(X,Y,\mathbb{P})}(x,y) \ d\mu_{XY}(x,y) = \mu(A) & & \forall A \in \beta_{\mathbb{R}^{2}}.
	\end{align}
	
	When there is no confusion about which random variables $L_{(X,Y,\mathbb{P})}$ refers to, it will be denoted simply by $L_{\mathbb{P}}$, or $L$ if there is also no doubt about under which probability measure it is defined. The Lift Function $L_{(X,Y,\mathbb{P})}$ is unique up to a $\mu_{XY}$ null set, is also $\mathcal{F}$-measurable and is well-defined as $\mu \ll \mu_{XY}$. From now on we consider $L_{(X,Y,\mathbb{P})}$ to be a \textit{version} of the Lift Function of $(X,Y)$ under $\mathbb{P}$.
	
	The Lift Function may also be defined in terms of the conditional probability function of $(X,Y)$, a $\beta_{\mathbb{R}}$-measurable function that, given a $B \in \beta_{\mathbb{R}}$, is denoted by $P(B \mid \cdot): \mathbb{R} \rightarrow \mathbb{R}_{+}$ and satisfies
	\begin{align}
	\label{conditional}
	\int\displaylimits_{A} P(B \mid x) \ d\mu_{X}(x) = \mu(A \times B) & & \forall A \in \beta_{\mathbb{R}}.
	\end{align}
	Note that $P(B \mid \cdot)$ is unique up to a $\mu_{X}$-null set (see \cite[Section~33]{billingsley2008} for more details) and is also $\mathcal{F}$-measurable. Furthermore, if we fix an $x \in \mathbb{R}$ we may take the $x$-section $L^{x}(y) \coloneqq L(x,y)$ of $L$ as a function of $y$ and then define it in terms of the conditional probability function $P(B \mid \cdot)$ for a $B \in \beta_{\mathbb{R}}$. 
	
	\begin{proposition}
		\label{P1}
		If $\mu \ll \mu_{XY}$, then for all $B \in \beta_{\mathbb{R}}$ 
		\[
		\int\displaylimits_{B} L^{x}(y) \ d\mu_{Y}(y) = P(B \mid x)
		\]
		$\mu_{X}$-almost every $x \in \mathbb{R}$.
	\end{proposition}
	\begin{proof}
		Let $\mathcal{B} = \{B_{1} \times B_{2}: B_{1}, B_{2} \in \beta_{\mathbb{R}}\} \subset \beta_{\mathbb{R}^{2}}$ be the set of all Borel rectangles. By Tonelli's Theorem, as $L \geq 0$, it follows that, for all $B_{1} \times B_{2} \in \mathcal{B}$,
		\begin{align*}
		\mu(B_{1} \times B_{2}) = & \int\displaylimits_{B_{1} \times B_{2}} L(x,y) \ d\mu_{XY}(x,y) = \int\displaylimits_{B_{1}} \Bigg[\int\displaylimits_{B_{2}} L^{x}(y) \ d\mu_{Y}(y)\Bigg] d\mu_{X}(x).
		\end{align*}
		Furthermore, by (\ref{conditional}) we have for all $B_{1}, B_{2} \in \beta_{\mathbb{R}}$ that
		\begin{align*}
		\int\displaylimits_{B_{1}} P(B_{2} \mid x) \ d\mu_{X}(x) = \mu(B_{1} \times B_{2})
		\end{align*}
		and it follows that for all $B_{1}, B_{2} \in \beta_{\mathbb{R}}$
		\begin{align*}
		\int\displaylimits_{B_{1}} \Bigg[\int\displaylimits_{B_{2}} L^{x}(y) \ d\mu_{Y}(y)\Bigg] d\mu_{X}(x) = \int\displaylimits_{B_{1}} P(B_{2} \mid x) \ d\mu_{X}(x)
		\end{align*}
		and $\int\displaylimits_{B_{2}} L^{x}(y) \ d\mu_{Y}(y) =  P(B_{2} \mid x)$ $\mu_{X}$-almost every $x \in \mathbb{R}$, for all $B_{2} \in \beta_{\mathbb{R}}$.
	\end{proof}
	
	The relation given in Proposition \ref{P1} provides us with a clue on the form of the Lift Function in some special cases, as when $\mu \ll \mathscr{L}_{2}$ and $\mu_{X}, \mu_{Y} \ll \mathscr{L}_{1}$, i.e., $(X,Y)$ is absolutely continuous w.r.t. Lebesgue measure (or when $\mu,\mu_{X}$ and $\mu_{Y}$ are absolutely continuous w.r.t. the corresponding counting measure). Indeed, the Lift Function will be given by the ratio between the joint probability density (probability function) of $(X,Y)$ and the product of their marginal probability densities (probability functions), when the ratio is defined. This fact may be established by substituting these ratios in (\ref{LF}). Furthermore, this representation shows that the Lift Function provides a local measure of dependence between $X$ and $Y$. Indeed, when $X$ and $Y$ are independent the Lift Function equals one $\mu$-almost surely.
	
	\begin{proposition}
		\label{P2}
		The random variables $X$ and $Y$ are independent if, and only if, $\mu \ll \mu_{XY}$ and $L_{(X,Y,\mathbb{P})} \equiv 1$ $\mu$-almost surely.
	\end{proposition}
	\begin{proof}
		Note that if $\mu \ll \mu_{XY}$ and $L \equiv 1$ $\mu_{XY}$-almost surely, then $L \equiv 1$ $\mu$-almost surely for $\{A \in \beta_{\mathbb{R}^{2}}: \mu(A) = 0\} \supset \{A \in \beta_{\mathbb{R}^{2}}: \mu_{XY}(A) = 0\}$, so that it is enough to show that $L \equiv 1$ $\mu_{XY}$-almost surely.
		
		$(\implies)$
		Suppose that $\mu \ll \mu_{XY}$ and $L_{(X,Y,\mathbb{P})} \equiv 1$ $\mu_{XY}$-almost surely. Then, for all $B_{1}, B_{2} \in \beta_{\mathbb{R}}$,
		\begin{align*}
		\mu(B_{1} \times B_{2}) = \int\displaylimits_{B_{1} \times B_{2}} 1 \ d\mu_{XY}(x,y) = \mu_{X}(B_{1})\mu_{Y}(B_{2})
		\end{align*}
		and $X$ and $Y$ are independent by definition.
		
		$(\impliedby)$
		Suppose that $X$ and $Y$ are independent. Then, for all $B_{1}, B_{2} \in \beta_{\mathbb{R}}$,
		\begin{align*}
		\int\displaylimits_{B_{1} \times B_{2}} \! \! \! \! \! L(x,y) \ d\mu_{XY}(x,y) & = \mu(B_{1} \times B_{2})  = \mu_{X}(B_{1})\mu_{Y}(B_{2})  = \int\displaylimits_{B_{1} \times B_{2}} \! \! \! \! \! 1 \ d\mu_{XY}(x,y)
		\end{align*}
		so that $L = 1$ $\mu_{XY}$-almost surely. Now, if $\mu$ is not absolutely continuous w.r.t. $\mu_{XY}$ then $X$ and $Y$ are dependent.
	\end{proof}
	
	Another important property of the Lift Function is that it cannot be greater than one, nor lesser than one, $\mu_{XY}$-almost surely.
	
	\begin{proposition}
		\label{P3}
		If $\mu \ll \mu_{XY}$, then $\mu_{XY}(\{(x,y) \in \mathbb{R}^{2}: L(x,y) > 1\}) < 1$ and $\mu_{XY}(\{(x,y) \in \mathbb{R}^{2}: L(x,y) < 1\}) < 1$.
	\end{proposition}
	\begin{proof}
		Suppose that $\mu_{XY}(\{(x,y) \in \mathbb{R}^{2}: L(x,y) > 1\}) = 1$. Then
		\begin{align*}
		1 = \mu(\mathbb{R}^{2}) = \int\displaylimits_{\mathbb{R}^{2}} L(x,y) \ d\mu_{XY}(x,y) > \int\displaylimits_{\mathbb{R}^{2}} 1 \ d\mu_{XY}(x,y) = 1.
		\end{align*}
		Analogously, we can show that $\mu_{XY}(\{(x,y) \in \mathbb{R}^{2}: L(x,y) < 1\})$ cannot be equal to one.
	\end{proof}
	
	From Propositions \ref{P2} and \ref{P3} we see that, if $\mu \ll \mu_{XY}$, either $X$ and $Y$ are independent and $L \equiv 1$ $\mu_{XY}$-almost surely, or there are $\mu_{XY}$-non-null sets where $X$ \textit{lifts} $Y$ and $\mu_{XY}$-non-null sets where $X$ \textit{inhibits} $Y$. Therefore, the lift is not a property of the whole \textit{distribution} of $(X,Y)$, but is rather a \textit{pointwise} property of it. Note that this property is not satisfied by Sibuya's function \cite{sibuya}, as it may be greater than one or lesser than one for all points in $\mathbb{R}^{2}$, so that $(X,Y)$ are positively quadrant dependent and negatively quadrant dependent, respectively (see \cite{lehmann1966} for more details).
	
	\begin{remark}
		\normalfont The Lift Function as defined in (\ref{LF}) may be extended to the case in which we have two real-valued random vectors $\boldsymbol{X} = (X_{1},\dots,X_{m_{1}})$ and $\boldsymbol{Y} = (Y_{1},\dots,Y_{m_{2}})$, defined in $(\Omega,\mathcal{F},\mathbb{P})$. Indeed, let $(\mathbb{R}^{m_{1}},\beta_{\mathbb{R}^{m_{1}}}), (\mathbb{R}^{m_{2}},\beta_{\mathbb{R}^{m_{2}}})$ and $(\mathbb{R}^{m_{1}+m_{2}},\beta_{\mathbb{R}^{m_{1}+m_{2}}})$ be measurable spaces, and let $\mu$ be a $(\mathbb{R}^{m_{1}+m_{2}},\beta_{\mathbb{R}^{m_{1}+m_{2}}})$-probability measure such that $\mu(A) = \mathbb{P}((\boldsymbol{X},\boldsymbol{Y}) \in A), \forall A \in \beta_{\mathbb{R}^{m_{1}+m_{2}}}$. Define $\mu_{\boldsymbol{X}}$, $\mu_{\boldsymbol{Y}}$ as $\mu_{\boldsymbol{X}}(A) = \mu(A \times \mathbb{R}^{m_{2}})$ and $\mu_{\boldsymbol{Y}}(B) = \mu(\mathbb{R}^{m_{1}} \times B), \forall A \in \beta_{\mathbb{R}^{m_{1}}}, B \in \beta_{\mathbb{R}^{m_{2}}}$. Finally, define $\mu_{\boldsymbol{XY}}$ as the only $(\mathbb{R}^{m_{1}+m_{2}},\beta_{\mathbb{R}^{m_{1}+m_{2}}})$-probability measure which satisfies $\mu_{\boldsymbol{XY}}(A \times B) = \mu_{\boldsymbol{X}}(A)\mu_{\boldsymbol{Y}}(B)$ for all $A \in \beta_{\mathbb{R}^{m_{1}}}, B \in \beta_{\mathbb{R}^{m_{2}}}$. Then, the Lift Function of $(\boldsymbol{X},\boldsymbol{Y})$ is defined as the Radon-Nikodym derivative of $\mu$ with respect to $\mu_{\boldsymbol{XY}}$. In the multidimensional case, the Lift Function assesses the local dependence between $\boldsymbol{X}$ and $\boldsymbol{Y}$, rather than the dependence between the variables inside each random vector.
	\end{remark}
	
	\section{A General Lift Function}
	\label{Sec3}
	
	The Lift Function is a powerful tool for analysing locally the dependence between two random variables, but is somewhat limited as it is restricted to random variables such that $\mu \ll \mu_{XY}$, which is a strong restraint. As discussed in the introduction, in some simple cases, as when $X$ is absolutely continuous w.r.t. Lebesgue measure, and $Y$ equals a smooth function of $X$ with probability $1$, we have that $\mu \not\ll \mu_{XY}$, and the Lift Function (\ref{LF}) is not well-defined. Nevertheless, we now develop an extension of (\ref{LF}) which holds for a larger class of random variables.
	
	Suppose that $\mu_{X}, \mu_{Y} \ll \mathscr{L}_{1}$ and denote $\delta_{\epsilon}(x,y), \epsilon > 0, (x,y) \in \mathbb{R}^{2}$, as 
	\begin{equation*}
	\delta_{\epsilon}(x,y) = \begin{cases}
	1/\mathscr{L}_{2}\big(B_{\epsilon}(0,0)\big), & \text{ if } (x,y) \in B_{\epsilon}(0,0) \\
	0, & \text{ otherwise }
	\end{cases}
	\end{equation*}
	in which $B_{\epsilon}(x,y)$ is the ball centred at $(x,y) \in \mathbb{R}^{2}$ with radius $\epsilon$. The convolution measure of $\mu$ and $\delta_{\epsilon}$ is given by
	\begin{equation*}
	\mu_{\epsilon}(A) \coloneqq \big(\mu * \delta_{\epsilon}\big)(A) = \int_{\mathbb{R}^{2}} \int_{\mathbb{R}^{2}} \mathds{1}\{(x^*+x',y^*+y') \in A\} \ d\delta_{\epsilon}(x^*,y^*) \ d\mu(x',y')
	\end{equation*}
	for all $A \in \beta_{\mathbb{R}^{2}}$. The measures $\{\mu_{\epsilon}: \epsilon > 0\}$ are absolutely continuous w.r.t. $\mathscr{L}_{2}$ and have the following Radon-Nikodym derivatives.
	
	\begin{lemma}
		\label{mu_epsolon_continuous}
		The convolution measures $\{\mu_{\epsilon}: \epsilon > 0\}$ are such that $\mu_{\epsilon} \ll \mathscr{L}_{2}$ and 
		\[
		\rho_{\epsilon}(x,y) \coloneqq \frac{d\mu_{\epsilon}}{d\mathscr{L}_{2}}(x,y) = \frac{\mu(B_{\epsilon}(x,y))}{\mathscr{L}_{2}(B_{\epsilon}(0,0))}
		\]
		is a version of the respective Radon-Nikodym derivative, for $(x,y) \in \mathbb{R}^{2}$.
	\end{lemma}
	\begin{proof}
		The convolution of $\delta_{\epsilon}$ and $\mu$ at a point $(x,y) \in \mathbb{R}^{2}$ is given by
		\begin{align*}
		\big(\mu * \delta_{\epsilon}\big)(x,y) & = \int_{\mathbb{R}^{2}} \int_{\mathbb{R}^{2}} \mathds{1}\{(x'+x^*,y'+y^*) = (x,y)\} \ d\delta_{\epsilon}(x^*,y^*) \ d\mu(x',y') \\
		& = \int_{\mathbb{R}^{2}} \delta_{\epsilon}(x-x',y-y') \ d\mu(x',y') \\
		& = \int_{\mathbb{R}^{2}} \frac{\mathds{1}\{(x',y') \in B_{\epsilon}(x,y)\}}{\mathscr{L}_{2}(B_{\epsilon}(0,0))} \ d\mu(x',y') \\
		& = \frac{\mu(B_{\epsilon}(x,y))}{\mathscr{L}_{2}(B_{\epsilon}(0,0))}.
		\end{align*}
		Therefore, it follows that, for all $A \in \beta_{\mathbb{R}^{2}}$,
		\begin{align*}
		\int_{A} & \frac{\mu(B_{\epsilon}(x,y))}{\mathscr{L}_{2}(B_{\epsilon}(0,0))} \ d\mathscr{L}_{2}(x,y) = \\
		& = \int_{A} \int_{\mathbb{R}^{2}} \int_{\mathbb{R}^{2}} \mathds{1}\{(x'+x^*,y'+y^*) = (x,y)\} \ d\delta_{\epsilon}(x^*,y^*) \ d\mu(x',y') \ d\mathscr{L}_{2}(x,y) \\
		& = \int_{\mathbb{R}^{2}} \int_{\mathbb{R}^{2}} \int_{A} \mathds{1}\{(x'+x^*,y'+y^*) = (x,y)\} \ d\mathscr{L}_{2}(x,y) \ d\delta_{\epsilon}(x^*,y^*) \ d\mu(x',y') \\
		& = \int_{\mathbb{R}^{2}} \int_{\mathbb{R}^{2}} \mathds{1}\{(x'+x^*,y'+y^*) \in A\} \ d\delta_{\epsilon}(x^*,y^*) \ d\mu(x',y') = \mu_{\epsilon}(A).
		\end{align*}
	\end{proof}
	
	In order to establish a Lift Function for a more general class of random variables, we need to define a function for $\mu$ that will behave as a \textit{probability density} of it w.r.t. Lebesgue measure, that holds even when $\mu \not\ll \mathscr{L}_{2}$, which is a consequence of $\mu \not\ll \mu_{XY}$ when $\mu_{X}, \mu_{Y} \ll \mathscr{L}_{1}$. Observe that when $\mu \ll \mathscr{L}_{2}$
	\begin{equation}
	\label{rho_epsilon_continuous}
	\rho(x,y) \coloneqq \frac{d\mu}{d\mathscr{L}_{2}}(x,y) = \lim\limits_{\epsilon \rightarrow 0} \frac{d\mu_{\epsilon}}{d\mathscr{L}_{2}}(x,y) = \lim\limits_{\epsilon \rightarrow 0} \frac{\mu(B_{\epsilon}(x,y))}{\mathscr{L}_{2}(B_{\epsilon}(0,0))}
	\end{equation}
	for almost all $(x,y)$, and when $\mu \not\ll \mathscr{L}_{2}$ this limit may diverge or not exist. However, if we multiply the Radon-Nikodym derivative inside the limit by a power of $\epsilon$, we may obtain a finite limit which will work as a \textit{density} for $\mu$ w.r.t. Lebesgue measure. Let
	\begin{equation*}
	l(x,y) = \Big\{ s \in (-\infty,2]: \lim\limits_{\epsilon \rightarrow 0} \epsilon^{2-s} \rho_{\epsilon}(x,y) = 0\Big\}
	\end{equation*}
	and define $s(x,y) = \sup l(x,y)$. If $l(x,y)$ is not empty, then it is an open or closed half-line whose endpoint is $s(x,y)$.
	
	\begin{lemma}
		\label{properties_lxy}
		The set $l(x,y)$ is a half-line that contains $(-\infty,0)$. In particular, if $\mu \ll \mathscr{L}_{2}$ then $(-\infty,2) \subset l(x,y)$.
	\end{lemma}
	\begin{proof}
		If $s \in l(x,y)$ and $s' < s$ then $s' \in l(x,y)$ as, for $0 < \epsilon < 1$,
		\begin{align*}
		0 \leq \epsilon^{2-s'} \rho_{\epsilon}(x,y) \leq &\epsilon^{2-s} \rho_{\epsilon}(x,y) \\
		& \implies  0 \leq \lim\limits_{\epsilon \rightarrow 0} \epsilon^{2-s'} \rho_{\epsilon}(x,y) \leq \lim\limits_{\epsilon \rightarrow 0} \epsilon^{2-s} \rho_{\epsilon}(x,y) = 0
		\end{align*}
		so $l(x,y)$ is a half-line of the form $(-\infty,l(x,y))$ or $(-\infty,l(x,y)]$. Furthermore,	
		\begin{align*}
		\epsilon^{2-s} \rho_{\epsilon}(x,y) \leq \frac{\epsilon^{2-s}}{\mathscr{L}_{2}(B_{\epsilon}(0,0))} = \frac{\epsilon^{2-s}}{\pi \epsilon^{2}}
		\end{align*}
		so that, if $s < 0$, then
		\begin{align*}
		\lim\limits_{\epsilon \rightarrow 0} \epsilon^{2-s} \rho_{\epsilon}(x,y) \leq \lim\limits_{\epsilon \rightarrow 0} \epsilon^{-s} = 0
		\end{align*}
		and $(-\infty,0) \subset l(x,y)$. Finally, if $\mu \ll \mathscr{L}_{2}$ then $s(x,y) = 2$ by (\ref{rho_epsilon_continuous}), and $l(x,y) = (-\infty,2]$ or $l(x,y) = (-\infty,2)$, whether $\rho(x,y) = 0$ or not. 
	\end{proof}
	
	A \textit{(false) density} of $\mu$ w.r.t. $\mathscr{L}_{2}$ may then be defined as
	\begin{equation}
	\label{limit_rho}
	\tilde{\rho}(x,y) = \lim\limits_{\epsilon \rightarrow 0} \epsilon^{2-s(x,y)} \rho_{\epsilon}(x,y)
	\end{equation}
	for $(x,y) \in \mathbb{R}^{2}$ such that this limit exists. Note that, if $\mu \ll \mathscr{L}_{2}$, then $\tilde{\rho} \equiv \rho$. Proceeding this way, denoting $\rho_{X} = \frac{d\mu_{X}}{d\mathscr{L}_{1}}$ and $\rho_{Y} = \frac{d\mu_{Y}}{d\mathscr{L}_{1}}$, which are well-defined as $\mu_{X}, \mu_{Y} \ll \mathscr{L}_{1}$, we define the Lift Function as
	\begin{equation}
	\label{LF_general}
	L_{(X,Y,\mathbb{P})}(x,y) = \frac{\tilde{\rho}(x,y)}{\rho_{X}(x) \rho_{Y}(y)}
	\end{equation}
	when $\tilde{\rho}(x,y)$ is defined. The Lift Function (\ref{LF_general}) reduces to (\ref{LF}) when $\mu \ll \mu_{XY}$. Under definition (\ref{LF_general}), the Lift Function is well-defined for all random variables such that $\mu \ll \mu_{XY}$ or such that $\mu_{X}, \mu_{Y} \ll \mathscr{L}_{1}$ and the limit (\ref{limit_rho}) exists for all $(x,y) \in \mathbb{R}^{2}$, which covers a large class of singular joint distributions with absolutely continuous, w.r.t. Lebesgue measure, marginal distributions.
	
	The following example derives the Lift Function for the case in which $\mu_{X}$ is absolutely continuous w.r.t. $\mathscr{L}_{1}$, and $Y = \varphi(X)$ with probability 1, in which $\varphi$ is a smooth real-valued function. The interpretation of the Lift Function in this example is an illustration of how it may be useful for studying variable dependence.
	
	\begin{example}
		\normalfont
		By definition (\ref{LF_general}) we have that, if $\mu_{X}, \mu_{Y} \ll \mathscr{L}_{1}$ and $\mathbb{P}(Y = \varphi(X)) = 1$, i.e., $\mu\big(\gamma\big) = 1, \gamma \coloneqq \{(x,y) \in \mathbb{R}^{2}: y = \varphi(x)\}$, for $\varphi \in C^{1}(\mathbb{R})$, then 
		\begin{align*}
		\tilde{\rho}(x,\varphi(x)) & = \lim\limits_{\epsilon \rightarrow 0} \epsilon^{2-s} \frac{\int_{Proj_{X}(\gamma \cap B_{\epsilon}(x,\varphi(x)))} \rho_{X}(x') \ dx'}{\pi \epsilon^{2}} \\
		& = \lim\limits_{\epsilon \rightarrow 0} \frac{\epsilon^{-s}}{\pi} \int_{a_{\epsilon}}^{b_{\epsilon}} \frac{\rho_{X}(t)}{\lVert \gamma'(t) \rVert} \ dt = \lim\limits_{\epsilon \rightarrow 0} \frac{\epsilon^{-s}}{\pi} \frac{\rho_{X}(\tilde{x})}{\lVert \gamma'(\tilde{x}) \rVert} (b_{\epsilon} - a_{\epsilon}) \\
		& = \lim\limits_{\epsilon \rightarrow 0} \frac{2\epsilon^{-s+1}\rho_{X}(\tilde{x})}{\pi\sqrt{1 + [\varphi'(\tilde{x})]^{2}}} = \frac{2\rho_{X}(x)}{\pi\sqrt{1 + [\varphi'(x)]^{2}}}
		\end{align*}
		if $s \coloneqq s(x,y) = 1$, in which the curve $\gamma$ is parametrized as $\gamma(t) = (t,\varphi(t)), t \in \mathbb{R}$; $(a_{\epsilon},b_{\epsilon})$ is such that $\{(t,\gamma(t)): a_{\epsilon} \leq t \leq b_{\epsilon}\} \subset \gamma \cap B_{\epsilon}(x,\varphi(x))$ and $b_{\epsilon} - a_{\epsilon} = o(2\epsilon)$; and $\tilde{x} \in Proj_{X}(\gamma \cap B_{\epsilon}(x,\varphi(x))) \xrightarrow{\epsilon \rightarrow 0} x$. Therefore, $L(x,y) = 0$ if $y \neq \varphi(x)$ and
		\begin{equation}
		\label{LF_functional}
		L(x,\varphi(x)) = \frac{2}{\pi \rho_{Y}(\varphi(x)) \sqrt{1 + [\varphi'(x)]^{2}}}
		\end{equation}
		
		The form of the Lift Function (\ref{LF_functional}) has nice properties, which one would expect a generalization of (\ref{LF}) to have. On the one hand, as in the case in which $X$ and $Y$ are random variables absolutely continuous w.r.t. counting measure and $\mathbb{P}(Y = \varphi(X)) = 1$, the density of $Y$ is in the denominator of the Lift Function: in that case we had the density of $Y$ w.r.t. the counting measure, while in (\ref{LF_functional}) the density of $Y$ w.r.t. $\mathscr{L}_{1}$. 
		
		On the other hand, the Lift Function (\ref{LF_functional}) depends on the derivative of $\varphi$, which would be expected for a local dependence quantifier. Indeed, suppose that $\mu_{X}([a,b]) = 1$, for some $a,b \in \mathbb{R}$, and that $\mathbb{P}(Y_{1} = \varphi_{1}(X)) = \mathbb{P}(Y_{2} = \varphi_{2}(X)) = 1, \varphi_{1},\varphi_{2} \in C^{1}([a,b])$. If $\varphi'_{1}(x) > \varphi'_{2}(x)$ for almost every $x$, then, in the first case, the probability mass is spread over a curve of length $\int_{a}^{b} \sqrt{1 + [\varphi'_{1}(t)]^{2}} \ dt > \int_{a}^{b} \sqrt{1 + [\varphi'_{2}(t)]^{2}} \ dt$, the length of the curve in which the probability mass is spread over in the second case.
		
		Therefore, the \textit{local dependence} between $X$ and $Y_{2}$ is \textit{greater} than between $X$ and $Y_{1}$, as the probability mass is \textit{most concentrated} in the joint distribution of $(X,Y_{2})$. This fact is expressed by (\ref{LF_functional}) as $L_{(X,Y_{1})} \leq L_{(X,Y_{2})}$ for almost every $x$, showing that the Lift Function portrays in detail the dependence between two random variables, and is capable of drawing a distinction between two joint distributions by representing the nuances of their (local) dependence.
		
		The Lift Function for this case may also be obtained by another method. Let $\nu_{1}$ be a $(\mathbb{R}^{2},\beta_{\mathbb{R}^{2}})$-measure such that
		\begin{equation*}
		\nu_{1}(A) = \int_{\gamma \cap A} \rho_{X}(x) \rho_{Y}(y) \ d\sigma(x,y)
		\end{equation*}
		for all $A \in \beta_{\mathbb{R}^{2}}$, in which the integral is the respective line integral where $d\sigma$ is the induced volume form on the surface $\gamma \cap A$. Then
		\begin{equation}
		\label{LF_line_integral}
		L(x,y) = \frac{2}{\pi} \frac{d\mu}{d\nu_{1}}(x,y)
		\end{equation}
		for all $(x,y) \in \mathbb{R}^{2}$ is a version of $L$. Indeed, $\mu \ll \nu_{1}$, $\nu_{1}$ is $\sigma$-finite and
		\begin{align*}
		\int_{A} \frac{1}{\rho_{Y}(\varphi(x)) \sqrt{1+ [\varphi'(x)]^{2}}} \ d\nu & = \! \! \! \! \! \! \! \! \! \! \int\limits_{Proj_{X}(\gamma \cap A)} \! \! \! \! \frac{\rho_{X}(x) \rho_{Y}(\varphi(x))}{\rho_{Y}(\varphi(x)) \sqrt{1+ [\varphi'(x)]^{2}}} \lVert \gamma'(x) \rVert \ dx \\
		& = \mu_{X}(Proj_{X}(\gamma \cap A)) = \mu(\gamma \cap A)
		\end{align*}
		for all $A \in \beta_{\mathbb{R}^{2}}$. Identity (\ref{LF_line_integral}) is evidence of a more general relation between the Lift Function and the Hausdorff measures in $\mathbb{R}^{2}$, which is explored in the next section.	
		\hfill$\square$
	\end{example}
	
	\begin{remark}
		\normalfont If $\mu_{X}, \mu_{Y} \ll \eta_{1}$ and $\mu_{XY} \ll \eta_{2}$, in which $\eta_{1}$ and $\eta_{2}$ are $\sigma$-finite measures in $(\mathbb{R},\beta_{\mathbb{R}})$ and $(\mathbb{R}^{2},\beta_{\mathbb{R}^{2}})$, respectively, then it may be possible to define the Lift Function analogously to (\ref{LF_general}), interchanging $\mathscr{L}_{1}, \mathscr{L}_{2}$ by $\eta_{1}, \eta_{2}$ and defining suitable $\delta_{\epsilon}$ measures. 
	\end{remark}
	
	\begin{remark}
		\normalfont
		Although the Lift Function (\ref{LF_general}) is well-defined for a large class of random variables, it may not be \textit{informative}, in the sense of expressing the local dependence between $X$ and $Y$. In some cases, it may be needed to adapt the definition of the Lift Function by, for example, defining a \textit{(false) density} for $\mu$ from another convolution, involving measures that approximate a measure other than the Dirac delta.
	\end{remark}
	
	\subsection{Local Lift Dependence and Hausdorff Measure}
	
	The $m$ Hausdorff density of a measure $\mu$ in $(\mathbb{R}^{2},\beta_{\mathbb{R}^{2}})$ is defined as
	\begin{equation}
	\label{hausdorff_density}
	D_{m}(x,y) \coloneqq \lim\limits_{\epsilon \rightarrow 0} \frac{\mu(B_{\epsilon}(x,y))}{\epsilon^{m}}
	\end{equation}
	when the limit exists (see \cite[Chapter~2]{falconer1986} and \cite[Chapter~5]{preiss1987} for more details). Therefore, it follows that
	\begin{equation*}
	L(x,y) = \frac{1}{\pi}D_{s(x,y)}(x,y)
	\end{equation*}
	when both are well-defined. Denoting $A_{s}, s \in \{1,2\}$, as a measurable set such that $s(x,y) = s$ for all $(x,y) \in A_{s}$, we have that if limit (\ref{hausdorff_density}) exists and is positive for all $(x,y) \in A_{s}$, then $\mu$ restricted to $A_{s}$ is absolutely continuous w.r.t. $\mathscr{H}^{s}|_{A_{s}}$, the $s$-Hausdorff measure in $\mathbb{R}^{2}$ restricted to $A_{s}$.
	
	\begin{proposition}
		\label{P_haurdof_ac}
		Let $A_{s} \in \beta_{\mathbb{R}^{2}}, s \in \{1,2\}$, be such that $s(x,y) = s$ for all $(x,y) \in A_{s}$ and limit (\ref{hausdorff_density}) is positive and finite $\mu$-almost every $(x,y) \in A_{s}$. Then $\mu|_{A_{s}} \ll \mathscr{H}^{s}|_{A_{s}}$.
	\end{proposition}
	\begin{proof}
		By Preiss's theorem \cite{preiss1987}, if $0 \leq s \leq 2$ is integer and
		\begin{equation*}
		0 < \lim\limits_{\epsilon \rightarrow 0} \frac{\mu(B_{\epsilon}(x,y))}{\epsilon^{s}} < \infty
		\end{equation*}
		$\mu$-almost every $(x,y) \in A_{s} \in \beta_{\mathbb{R}^{2}}$, then $\mu|_{A_{s}} \ll \mathscr{H}^{s}|_{A_{s}}$ and $\mu$-almost every $A_{s}$ can be covered by countably many $s$-dimensional submanifolds of class one of $\mathbb{R}^{2}$, i.e., $\mu|_{A_{s}}$ is $s$-rectifiable (see \cite[5.1]{preiss1987} for more details).
	\end{proof}
	
	\begin{remark}
		\normalfont
		Proposition \ref{P_haurdof_ac} also holds for $s = 0$. However, there is no $A_{0} \in \beta_{\mathbb{R}^{2}}$ satisfying the conditions of the proposition as
		\begin{align*}
		0 < \lim\limits_{\epsilon \rightarrow 0} \frac{\mu(B_{\epsilon}(x,y))}{\pi} = \frac{\mu(x,y)}{\pi}
		\end{align*}
		implies $\min\{\mu_{X}(x),\mu_{Y}(y)\} \geq \mu(x,y) > 0$, which cannot be as $\mu_{X}$ and $\mu_{Y}$ are absolutely continuous w.r.t. $\mathscr{L}_{1}$.
	\end{remark}
	
	\begin{remark}
		\label{remark_sxy}
		\normalfont
		If we extend the definition of $A_{s}$ for $s \in [0,2]$ we note that there is no $A_{s}, s \in [0,1),$ such that $\mu(A_{s}) > 0$. Indeed, if $s \in [0,1)$, then $\mathscr{H}^{1}(A_{s}) = 0$. Furthermore, we have that 
		\begin{equation*}
		\mathscr{L}_{1}(Proj_{X}(A_{s})) \leq \mathscr{H}^{1}(A_{s}) = 0
		\end{equation*}
		so that $\mu_{X}(Proj_{X}(A_{s})) = 0$, as $\mu_{X} \ll \mathscr{L}_{1}$ (see \cite[Lemma~6.1]{falconer1986} for more details). But this implies that $\mu(A_{s}) \leq \mu_{X}(Proj_{X}(A_{s})) = 0$. This fact yields $1 \leq s(x,y) \leq 2$ $\mu$-almost every $(x,y) \in \mathbb{R}^{2}$.
	\end{remark}
	
	If the hypothesis of Proposition \ref{P_haurdof_ac} hold and $\mathscr{H}^{s}|_{A_{s}}$ is $\sigma$-finite, then the Lift Function is proportional to the Radon-Nikodym derivative between $\mu$ and the measure generated by the integral of $\int \rho_{X}\rho_{Y} \ d\mathscr{H}^{s}|_{A_{s}}$. A special case of this fact is (\ref{LF_line_integral}).
	
	\begin{theorem}
		\label{theorem1_hausdorff}
		Let $A_{s} \in \beta_{\mathbb{R}^{2}}, s \in \{1,2\}$, be such that $s(x,y) = s$ for all $(x,y) \in A_{s}$ and limit (\ref{hausdorff_density}) is positive and finite $\mu$-almost every $(x,y) \in A_{s}$. If $\mathscr{H}^{s}|_{A_{s}}$ is $\sigma$-finite, then for all $(x,y) \in A_{s}$ and a constant $\alpha(s)$,
		\begin{equation*}
		L(x,y) = \alpha(s) \frac{d\mu|_{A_{s}}}{d\nu_{s}}(x,y)
		\end{equation*}
		in which $\nu_{s}$ is $(A_{s},\beta_{\mathbb{R}^{2}}|_{A_{s}})$-measurable and
		\begin{equation}
		\label{nu}
		\nu_{s}(B) = \int_{B} \rho_{X}(x)\rho_{Y}(y) \ d\mathscr{H}^{s}|_{A_{s}}(x,y)
		\end{equation}
		for $B \in \beta_{\mathbb{R}^{2}}|_{A_{s}}$.
	\end{theorem}
	\begin{proof}
		If $s = 2$ then $\nu \sim \mu_{XY}$, i.e., $\nu$ is equivalent to $\mu_{XY}$, and the result follows from (\ref{rho_epsilon_continuous}) with $\alpha(2) = 1$, as $\mu|_{A_{2}} \ll \mu_{XY}|_{A_{2}}$ and the Lift Function reduces to (\ref{LF}). If $s = 1$, then, by Preiss' theorem \cite{preiss1987}, there exists a countable cover $\{\gamma_{n}\}_{n\geq1}$ of $A_{1}$, of $1$-dimensional submanifolds of class one of $\mathbb{R}^{2}$. Denote $\gamma_{n} \coloneqq \{(t,\varphi_{n}(t)):t \in I_{n} \in \beta_{\mathbb{R}}\}$, in which $\varphi_{n} \in C^{1}(I_{n})$, and suppose that $\{\gamma_{n}\}_{n\geq1}$ are $\mu$-almost disjoint, i.e., $\mu(\gamma_{i} \cap \gamma_{j}) = 0, i \neq j$. Then, for $(x,y) \in \gamma_{n}$, analogously to (\ref{LF_functional}), we have that
		\begin{align}
		\label{rho_general_hausdorff1} \nonumber
		\tilde{\rho}(x,\varphi_{n}(x)) & = \lim\limits_{\epsilon \rightarrow 0} \epsilon^{2-s} \frac{\int_{Proj_{X}(\gamma_{n} \cap B_{\epsilon}(x,\varphi_{n}(x)))} \rho_{X}(x') J_{x'} \ dx'}{\pi \epsilon^{2}} \\
		& = \frac{2J_{x}\rho_{X}(x)}{\pi\sqrt{1 + [\varphi_{n}'(x)]^{2}}}
		\end{align}
		in which $J_{x'}$ is the Jacobian of the change of variables, i.e., such that the integral on the projection equals $\mu(B_{\epsilon}(x,y) \cap \gamma_{n})$. Now, for $B \in \beta_{\mathbb{R}^{2}}|_{A_{1}}$ and $\alpha(1) = 2/\pi$, (\ref{rho_general_hausdorff1}) yields  
		\begin{align*}
		\int_{B} (\pi/2) L(x,y) \ d\nu_{1}(x,y) & = \int_{B} (\pi/2) L(x,y) \rho_{X}(x)\rho_{Y}(y) \ d\mathscr{H}^{1}|_{A_{1}}(x,y)\\
		& = \sum_{n \geq 1} \int_{B \cap \gamma_{n}} (\pi/2) L(x,y) \rho_{X}(x)\rho_{Y}(y) \ d\sigma(x,y) \\
		& = \sum_{n \geq 1} \int_{Proj_{X}(B \cap \gamma_{n})} \rho_{X}(x) J_{x} \ dx \\
		& = \sum_{n \geq 1} \mu(B \cap \gamma_{n}) = \mu(B).
		\end{align*}
	\end{proof}
	
	Theorem \ref{theorem1_hausdorff} characterizes the Lift Function for a large class of joint distribution. Indeed, by Lebesgue's Decomposition Theorem (see \cite[p.~425]{billingsley2008} for more details), we may write $\mu = \mu_{1} + \mu_{2}$ in which $\mu_{1} \ll \mu_{XY}$ and $\mu_{2} \perp \mu_{XY}$, and, if the Hausdorff dimension of the support of $\mu_{2}$ $A_{1} \coloneqq supp \ \mu_{2}$, is one, then we may write
	\begin{equation*}
	L(x,y) = \frac{d\mu}{d\mu_{XY}} \mathds{1}\{(x,y) \in A_{1}^{c}\} + \frac{d\mu}{d\nu_{1}} \mathds{1}\{(x,y) \in A_{1}\}
	\end{equation*}
	which is defined for all $(x,y) \in A^{c}_{1}$ and all $(x,y) \in A_{1}$ such that the second Radon-Nikodym derivative $d\mu/d\nu_{1}(x,y)$ is well-defined, in which $\nu_{1}$ is given by (\ref{nu}). Therefore, unless $supp \ \mu_{2}$ is a \textit{fractal set}, i.e., has a non-integer Hausdorff dimension, or $A_{1}$ is such that $\mathscr{H}^{1}|_{A_{1}}$ is not $\sigma$-finite, the Lift Function is well-defined and can be calculated by means of line integrals. 
	
	We now consider more examples of the general Lift Function.
	
	\begin{example}
		\normalfont
		Let $\{\varphi_{n}\}_{n \geq 1}$ be such that $\varphi_{n} \in C^{1}(\mathbb{R})$ and if $i \neq j$ then $\{x \in \mathbb{R}: \varphi_{i}(x) = \varphi_{j}(x)\}$ is at most countable. Also, let $\{a_{n}\}_{n \geq 1}$ be such that $0 \leq a_{n} \leq 1$ and $\sum_{n \geq 1} a_{n} = 1$. Finally, let $X$ and $Y$ be absolutely continuous random variables, defined in $(\Omega,\mathcal{F},\mathbb{P})$, such that the conditional probability function of $Y$ given $X$ is
		\begin{equation*}
		\mathbb{P}(Y = y | X = x) = \begin{cases}
		a_{n}, & \text{ if } y = \varphi_{n}(x), n \geq 1 \\
		0, & \text{ otherwise}
		\end{cases}.
		\end{equation*}
		Then, by applying Theorem \ref{theorem1_hausdorff} we have that $L(x,y) = 0$ if $y \neq \varphi_{n}(x), \forall n \geq 1,$ and
		\begin{equation*}
		L(x,y) =  \frac{2a_{n}}{\pi \rho_{Y}(\varphi_{n}(x)) \sqrt{1 + [\varphi_{n}'(x)]^{2}}}
		\end{equation*}
		if $y = \varphi_{n}(x), n \geq 1$. Note that, if $y = \varphi_{n}(x)$ for more than one $n \geq 1$, then we may choose a \textit{version} of $L$ by choosing the constant $a_{n}$.
		\hfill$\square$
	\end{example}
	
	\begin{example}
		\normalfont
		We now consider an example in which there exists $s \in (1,2)$ such that $\mu(A_{s}) = 1$, i.e., the support of $\mu$ is a \textit{fractal} set. Let $X$ be absolutely continuous w.r.t. $\mathscr{L}_{1}$, $\mathbb{P}(X \in [0,1/2]) = 1$, and define $Y$ such that $\mathbb{P}(Y = w(X)) = 1$ in which
		\begin{align*}
		w(x) = \sum_{n=1}^{\infty} \frac{cos(2\pi3^{n}x)}{2^{n}} & & x \in [0,1/2].
		\end{align*}
		Function $w$, known as \textit{Weierstrass function}, was first proposed by \cite{weierstrass1895} (see \cite{edgar1993} for a translated version of \cite{weierstrass1895}) as an example of a continuous function nowhere differentiable. The graph of $w$  is presented in Figure \ref{w_function}.	
		
		From the above definition of $(X,Y)$ we have that the support of $\mu$ is the set $\gamma \coloneqq \{(x,w(x)): x \in [0,1/2]\}$, which is a \textit{fractal} set, i.e., its Hausdorff dimension is a non-integer between $1$ and $2$. In fact, the actual Hausdorff dimension of $\gamma$ was unknown until recently, when \cite{shen2018} showed that it is indeed $2 - \log 2/\log 3 \approx 1.369$, which was a long standing conjecture. Note that, even though $Y$ is also absolutely continuous w.r.t. $\mathscr{L}_{1}$, there is no direct way of calculating the Lift Function (\ref{LF_general}) of $(X,Y)$ by the methods presented in this paper, i.e., there is no straightforward manner of calculating the limit (\ref{limit_rho}).	
		\hfill$\square$
		
		\begin{figure}
			\vspace{6pc}
			\includegraphics[width = \linewidth]{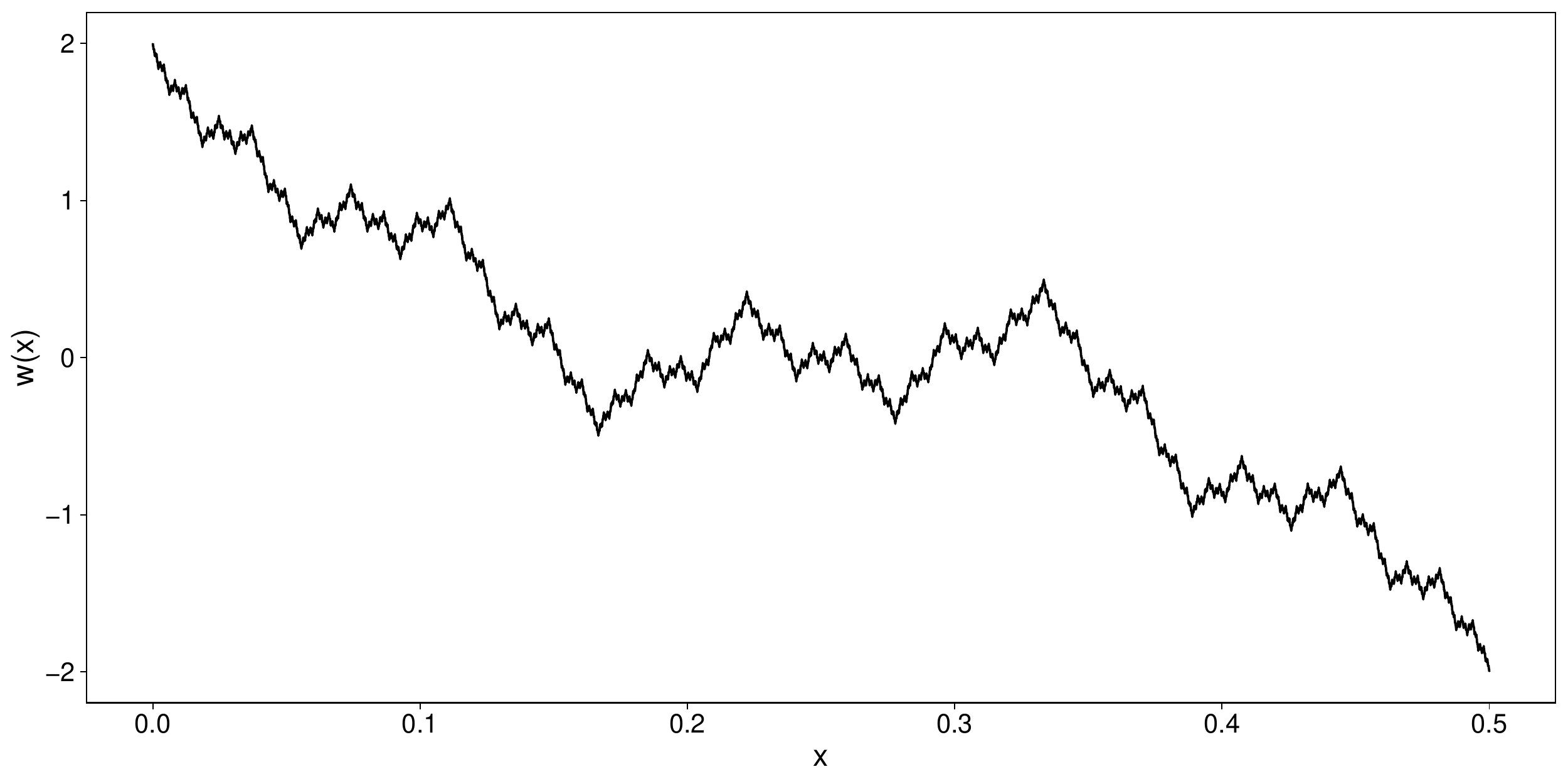}
			\caption[]{Weierstrass function for $x \in [0,1/2]$.}
			\label{w_function}
		\end{figure}
	\end{example}
	
	\section{Mutual Information}
	\label{Sec4}
	
	Although the Lift Function gives a high \textit{resolution}, wide and pointwise view of the form of dependence between two random variables, it may be of interest to assess their dependence \textit{globally}. A global quantifier of dependence is given by the Mutual Information
	\begin{equation}
	\label{MI}
	I(X,Y) = \int_{\mathbb{R}^{2}} \log L_{(X,Y,\mathbb{P})}(x,y) \ d\mu(x,y)
	\end{equation}
	when the Lift Function is defined. When $\mu \ll \mu_{XY}$, the Mutual Information is the mean deviation (in the logarithm scale) of the conditional distribution of $Y$ given $X$ and the marginal distribution of $Y$, quantifying the dependence between the variables. When $\mu \not\ll \mu_{XY}$ it may not be clear what $I(X,Y)$ means, for the following reasons.
	
	First of all, in some cases a limit cannot be interchanged with the integral in (\ref{MI}). For example, suppose that $\{(X_{n},Y_{n})\}_{n \geq 1}$ is a sequence of random vectors such that $(X_{n},Y_{n})$ has the Standard Bivariate Normal Distribution with correlation $r_{n} \xrightarrow{n \rightarrow \infty} 1$. Then
	\begin{equation*}
	I(X_{n},Y_{n}) = -\frac{1}{2} \log \Big(1 - r_{n}^{2}\Big) \xrightarrow{n \rightarrow \infty} + \infty
	\end{equation*}
	but $(X_{n},Y_{n}) \xrightarrow{\mathcal{L}} (X,Y)$, with $X \sim N(0,1)$ and $\mathbb{P}(Y = X) = 1$, so that, by (\ref{LF_functional}),
	\begin{align*}
	I(X,Y) & = \log \frac{\sqrt{2}}{\pi} - \int_{\mathbb{R}} \log\big(\rho_{X}(x)\big) \rho_{X}(x) \ dx \\
	& = \log \frac{\sqrt{2}}{\pi} + \log(\sqrt{2\pi e}) \\
	& = \log \frac{2\sqrt{e}}{\sqrt{\pi}}	
	\end{align*}
	and therefore $\lim\limits_{n \rightarrow \infty} I(X_{n},Y_{x}) \neq I\big(\lim\limits_{n \rightarrow \infty} X_{n},\lim\limits_{n \rightarrow \infty} Y_{n}\big)$, in which the limits in the right-hand side is in law. Indeed, not even the limit of the Lift Function is the Lift Function of the limit as
	\begin{equation*}
	L_{(X_{n},Y_{n})} = (1-r_{n}^{2})^{-1/2} \exp\Big(- \frac{1}{2(1 - r_{n}^{2})} \big[ x^{2} + y^{2} - 2r_{n}xy\big] + \frac{x^{2} + y^{2}}{2} \Big)
	\end{equation*}
	which converges to zero if $x \neq y$, but diverges when $x = y$, as $n \rightarrow \infty$.
	
	Hence, the Mutual Information is in general not a suitable scale for \textit{comparing} the dependence inside two random vectors in general scenarios. Of course, it has a lot of qualities, especially when dealing with random variables absolutely continuous w.r.t. counting measure, which is evident by the extend and reach of its applications, specially in Information Theory. Nevertheless, when calculated in more general setups, as when the variables are absolutely continuous w.r.t. Lebesgue measure, it is not an appropriate scale to compare the dependence inside random vectors, as random vectors \textit{more dependent} in a natural way, e.g, when one variable is a function of the other, may have a smaller Mutual Information than random vectors which are not \textit{as dependent} in some sense.
	
	On the other hand, when we compare the dependence inside random vectors by comparing their Lift Function we capture in more detail the nuances of it. This may be done by observing regions in which the Lift Function is greater than $1$ (lift regions) and regions in which it is lesser than $1$ (inhibition regions). The \textit{patterns} observed in a Lift Function present in more detail the form of dependence between random variables than a global quantifier as the Mutual Information.
	
	As an example, suppose that $(X_{1},Y_{1})$ follows a standard Bivariate Normal Distribution with correlation $0.6$ and $(X_{2},Y_{2})$ follows a Circular Bivariate Cauchy Distribution (see \cite{ferguson1962} for more details), whose Lift Functions are represented by their \textit{contours} and \textit{heatmaps} in Figure \ref{heatmaps}. On the one hand, we see that the \textit{lift regions} of $(X_{1},Y_{1})$ are concentrated around the line with slope $1$ and intercept $0$. On the other hand, the \textit{lift regions} of $(X_{2},Y_{2})$ are outside two hyperbolas.
	
	\begin{figure}[ht]
		\vspace{6pc}
		\includegraphics[width = \linewidth]{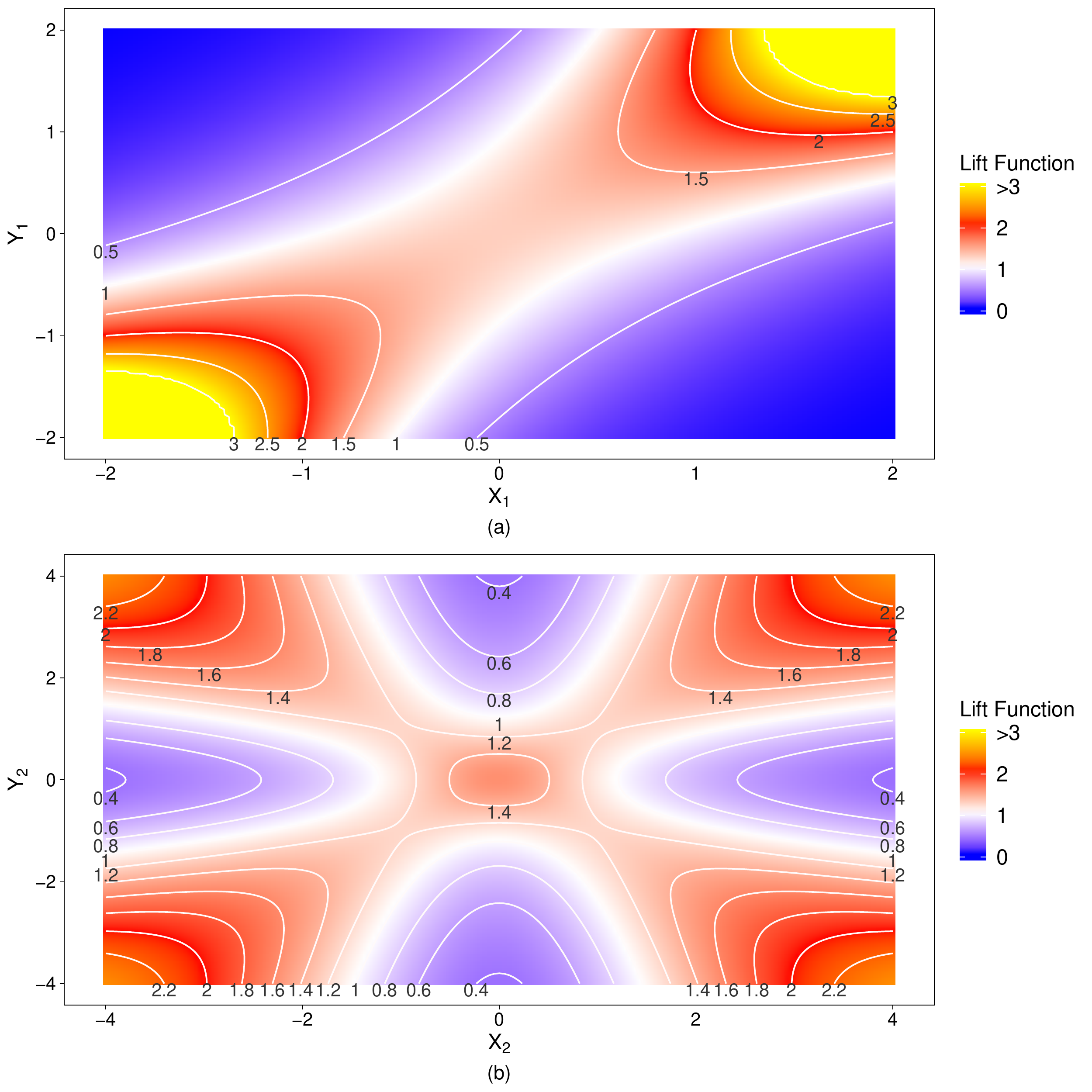}
		\caption[]{\footnotesize (a) Heatmap of the Lift Function of a standard Bivariate Normal Distribution with correlation $0.6$. (b) Heatmap of the Lift Function of a Circular Bivariate Cauchy Distribution, i.e., the bivariate distribution with joint density $f(x,y) = 1/2\pi(1 + x^{2} + y^{2})^{3/2}$ for $(x,y) \in \mathbb{R}^{2}$.}
		\label{heatmaps}
	\end{figure}
	
	From the Lift Functions we see that the dependence between $X_{1}$ and $Y_{1}$ is \textit{linear}, as $\{X_{1} = x\}$ lifts the event $\{Y_{1} = x \pm \epsilon\}$, for a $\epsilon > 0$, as the region around the line with slope $1$ and intercept $0$ has Lift Function greater than one. However, the dependence between $X_{2}$ and $Y_{2}$ has more nuances as their Lift Function has a more complex pattern. For example, $\{X = x\}$, for $|x| > 2$, lifts the event $\{Y \in A\}$ in which $A = [-4,-4+\epsilon] \cup [4-\epsilon,4]$ for $\epsilon > 0$, so the  \textit{four corners} of $[-4,4]^{2}$ are \textit{lift regions} of $(X_{2},Y_{2})$; on the other hand, $\{X = x\}$, for $|x| < 1$, lifts the event $\{Y \in B\}$ in which $B$ is an interval centred at the origin.
	
	Even though $I(X_{1},Y_{1}) = I(X_{2},Y_{2}) = 0.223$, the form of dependence between $X_{1}$ and $Y_{1}$, and between $X_{2}$ and $Y_{2}$ are quite different, which is evidence that, in order to really \textit{understand} the dependence between random variables we should study it locally, instead of globally: and the Lift Function is an useful tool for such study, as it may represent \textit{multiple forms} of dependence, as evidenced in Figure \ref{heatmaps}.

	\section{Final Remarks}
	\label{Sec5}
	
	This paper defines a local and general dependence quantifier, the Lift Function, which may be applied to asses the dependence between two random variables in rather general cases. When comparing with other local dependence quantifiers in the literature (see \cite{bairamov,bjerve,holland,jones1996,lehmann1966,sibuya} for example), the Lift Function is not restricted to the study of a specific form of dependence (as linear \cite{bairamov} and quadrant \cite{lehmann1966,sibuya} dependence); does not need the notion of \textit{regression curve}, which is considered by \cite{bjerve,jones1996} to study the dependence that may summarized as the ``proportion of variance explained by regression curve''; it may be applied to a large class of random variables, as opposed to \cite{holland} for example; and it may be more straightforward to interpret the relation between the variables by studying the patterns of the Lift Function, as can be established from Figure \ref{heatmaps}.
	
	From an applied point of view, the Lift Function is quite relevant in statistics in the scenario in which we apply a \textit{treatment} to a sample, observe an \textit{answer} and want to \textit{maximize} the expected number of specific answers in the sample. For example, suppose that the answer for our treatment is a random variable $Y$ absolutely continuous w.r.t. counting measure and that we want to maximize the frequency of an specific answer in the sample, say $\{Y = 1\}$. On the one hand, if we apply the treatment to a random sample of size $n$ we will expect $n \times \mathbb{P}(Y = 1)$ desired answers. On the other hand, suppose there is another random variable $X$, also absolutely continuous w.r.t. counting measure, that represents the \textit{profile} of the sample units and that $L(x,y)$ is known. Then, in order to maximize the number of desired answers in the sample we may apply the treatment not to \textit{any} $n$ sample units, but rather to $n$ sample units with profile $x_{opt} = \arg \max\limits_{x \in \mathbb{R}} L(x,1)$. In this framework, instead of expecting $n \times \mathbb{P}(Y = 1)$ desired answers, we will expect $n \times \mathbb{P}(Y = 1 \mid X = x_{opt})$, which is $[L(x_{opt},1) - 1] \times n$ more answers. This example may be extended to the continuous case, in which we want to maximize the answers in a subset of $\mathbb{R}$ and may choose profiles also in a subset of $\mathbb{R}$. For an application of the Lift Function in statistics see \cite{marcondes2018}.
	
	Nevertheless, the Lift Function is of great relevance not only for market segmentation or identification of prone individuals in a population, but also for the study of the dependence between two variables. Indeed, from the patterns of the Lift Function, it is possible to analyse the \textit{raw} dependence between variables, without restricting it to a specific kind of dependence, nor making assumptions about it. 
	
	We leave a few compelling topics for future research. From a statistical standpoint, besides the application of the Lift Function in specific cases, it would be interesting to develop estimation techniques for the Lift Function for the case in which $X$ and $Y$ are absolutely continous w.r.t. counting measure, and when their joint distribution is absolutely continuous w.r.t. Lebesgue measure. From a more theoretical standpoint, it would be interesting to study the Lift Function for the case in which $X$ and $Y$ are absolutely continuous, but the support of the singular part of their joint distribution is a \textit{fractal set}. Also, one could study the cases in which the Lift Function is not defined in order to propose a more general definition to it. Finally, it could also be possible to model the dependence between random variables by modelling their Lift Function.
	
	We believe that this paper contributes to the state-of-art of variable dependence assessment, proposing a quite general local dependence quantifier which is applicable to a large class of joint probability distributions. We believe that there are much more facets to the Lift Function which should be explored, and that it could be of use not only to assess variable dependence, but for applications in areas such Stochastic Processes and Machine Learning (see \cite{marcondes2018} for example).
	
	\section*{Acknowledgements}
	We would like to thank D. Tausk for some suggestions which improved our paper, and J. Barrera for the partnership in the applied research of the Lift Function which originated this abstract theoretical presentation of it.
		
	\bibliographystyle{plain}      
	\bibliography{Ref}   
	
\end{document}